\documentclass{amsart}
\usepackage{amsmath,amsfonts,amssymb,amsthm}
\usepackage{dsfont}
\usepackage{url}
\usepackage[colorlinks=true]{hyperref}
\usepackage{pdfsync}

\theoremstyle{plain}
\newtheorem{theorem}{Theorem}
\newtheorem{lemma}{Lemma}
\newtheorem{proposition}{Proposition}

\newtheorem*{thm*}{Theorem}

\theoremstyle{remark}
\newtheorem{remark}{Remark}
\newtheorem*{remark*}{Remark}

\theoremstyle{definition}
\newtheorem*{definition*}{Definition}
\newtheorem{definition}{Definition}

\newtheorem{example}{Example}
\newtheorem*{problem*}{Problem}

\def\C{\mathds{C}}
\def\N{\mathds{N}}
\def\R{\mathds{R}}
\def\Z{\mathds{Z}}
\def\n{\nu}

\def\PDE{{PDE}}

\def\e{\mathrm{e}}
\def\T{\mathds{T}}

\def\one{{\mathds{1}}}
\def\Hn{\mathcal{H}^{d-1}}

\newcommand{\Krand}{K_{\mathrm{rand}}(\mathrm{e})}
\newcommand{\Crand}{C_{\mathrm{rand}}(\mathrm{e})}
\newcommand{\Cdet}{C_{\mathrm{det}}}

\DeclareMathOperator*{\supp}{supp}
\DeclareMathOperator*{\ran}{ran}

\usepackage{xcolor}

\begin{document}

\title{On the randomised stability constant for inverse problems}

\author{Giovanni S.\ Alberti}

\address{MaLGa Center, Department of Mathematics, University of Genoa, Via Dodecaneso 35,
16146 Genova, Italy.}

\email{giovanni.alberti@unige.it}

\author{Yves Capdeboscq}

\address{Université de Paris, CNRS, Sorbonne Université, Laboratoire Jacques-Louis
Lions, UMR 7598, Paris, France}

\email{yves.capdeboscq@sorbonne-universite.fr}

\author{Yannick Privat}

\address{IRMA, Université de Strasbourg, CNRS UMR 7501, 7 rue René Descartes,
67084 Strasbourg, France}

\email{yannick.privat@unistra.fr}

\subjclass[2010]{65J22, 35R30}
\begin{abstract}
In this paper we introduce the randomised stability constant for abstract
inverse problems, as a generalisation of the randomised observability
constant, which was studied in the context of observability inequalities
for the linear wave equation. We study the main properties of the
randomised stability constant and discuss the implications for the
practical inversion, which are not straightforward. 
\end{abstract}

\keywords{Inverse problems, observability constant, compressed sensing, passive
imaging, regularisation, randomisation, deep learning, electrical
impedance tomography.}
\maketitle

\section{Introduction\label{sec:intro}}

Inverse problems are the key to all experimental setups where the
physical quantity of interest is not directly observable and must
be recovered from indirect measurements. They appear in many different
contexts including medical imaging, non-destructive testing, seismic
imaging or signal processing. In mathematical terms, an inverse problem
consists in the inversion of a linear or nonlinear map 
\[
T\colon X\to Y,\qquad x\mapsto T(x),
\]
which models how the quantity of interest $x$ belonging to a space
$X$ is related to the measurements $y=T(x)$ in the space $Y$. The
reader is referred to the many books on inverse problems for a comprehensive
exposition (see, e.g., \cite{isakov-2006,ammari-2008,tarantola-2009,2011-kirsch,HANDBOOK-MMI2-2015,alberti-capdeboscq-2016,2017-hasanov-romanov,ammari-2017}).

Inverse problems can be ill-posed: the map $T$ may not be injective
(i.e., two different $x_{1}$ and $x_{2}$ may correspond to the same
measurement $T(x_{1})=T(x_{2})$) or, when injective, $T^{-1}\colon\ran T\subseteq Y\to X$
may not be continuous (i.e., two different and not close $x_{1}$
and $x_{2}$ may correspond to almost identical measurements $T(x_{1})\approx T(x_{2})$).
Various strategies have been introduced to tackle the issue of inversion
in this setting, Tykhonov regularisation being the most famous method
\cite{1996-regularization,2008-kaltenbacher-neubauer-scherzer}.

Our purpose is to investigate the role of randomisation in the resolution
of inverse problems. By randomisation, we refer to the use of random
measurements, or to the case of random unknowns. We do not mean the
well-established statistical approaches in inverse problems, as in
the Bayesian framework, where probability is used to assess the reliability
of the reconstruction.

Even with this distinction, the wording ``randomisation'' may refer
to many different concepts in the framework of inverse problems. A
huge literature is devoted to the issue of randomisation in the measuring
process: the unknown $x$ in $X$ is fixed and deterministic, and
we choose the measurements randomly according to some suitable distribution.
For example, compressed sensing \cite{FR} and passive imaging with
ambient noise \cite{garnier-2016} belong to this class. Another very
popular idea consists in randomising the unknown: in this scenario,
we try to recover \textit{most} unknowns $x$ in $X$, according to
some distribution. For example, this is typically the situation when
deep learning \cite{Goodfellow-et-al-2016} is applied to inverse
problems. We briefly review these instances in Appendix~\ref{sec:rev}.

This article deals with a more recent approach, in the framework of
observability or control theory: the so-called randomised observability
constants \cite{PTZ1,PTZparab,PTZobsND}. We argue that, in contrast
with the initial motivation for its introduction, the randomised observability
constant is not necessarily indicative of the likelihood for randomised
unknowns to be observable.

In the next section, we recall the notion of randomised observability
constant and comment on its use in optimal design problems. In section~\ref{sec:randomAbstract},
we reformulate the randomised observability constant in an abstract
setting as randomised stability constant. In section~\ref{sec:propcstab}
we show that when the classical (deterministic) stability constant
is null, a positive randomised stability constant need not imply,
as one could hope, that the inverse problem can be solved for most
unknowns. In the course of our study, we make several observations
on the properties of the randomised stability constant. Section~\ref{sec:conclusion}
contains several concluding remarks and discusses possible future
directions.

\section{Randomising initial data of PDEs\label{Sec:revML}}

In this section we briefly review the randomised observability constant
introduced in \cite{PTZ1,PTZparab,PTZobsND} and its main properties.
This was motivated by the original idea of randomisation by Paley
and Zygmund, which we now briefly discuss.

\subsection{On randomisation processes}

\label{sec:paley} In order to understand the principle of randomisation
that will be at the core of this paper, it is useful to recall the
historical result by Paley and Zygmund on Fourier series. Let $(c_{n})_{n\in\Z}$
be an element of $\ell^{2}(\C)$ and $f$ be the Fourier series given
by 
\[
f:\T\ni\theta\mapsto\sum_{n\in\Z}c_{n}e^{in\theta},
\]
where $\T$ denotes the torus $\R/(2\pi)$. According to the so-called
Parseval identity, the function $f$ belongs to $L^{2}(\T)$; furthermore,
the coefficients $c_{n}$ can be chosen in such a way that $f$ does
not belong to any $L^{q}(\T)$ for $q>2$. Some of the results obtained
by Paley and Zygmund (see \cite{paley_zygmund_1930,paley_1930,paley_zygmund_1932})
address the regularity of the Fourier series $f$. They show that
if one changes randomly and independently the signs of the Fourier
coefficients $c_{n}$, then the resulting random Fourier series belongs
almost surely to any $L^{q}(\T)$ for $q>2$. More precisely, introducing
a sequence $(\beta_{n}^{\nu})_{n\in\Z}$ of independent Bernoulli
random variables on a probability space $(A,\mathcal{A},\mathds{P})$
such that 
\[
\mathds{P}(\beta_{n}^{\nu}=\pm1)=\frac{1}{2},
\]
then, the Fourier series $f^{\nu}$ given by 
\[
f^{\nu}:\T\ni\theta\mapsto\sum_{n\in\Z}\beta_{n}^{\nu}c_{n}e^{in\theta}
\]
belongs almost surely to $L^{q}(\T)$ for all $q<+\infty$.

In \cite{BurqTzvetkov}, the effect of the randomisation on the initial
data of solutions to dispersive equations has been investigated. In
that case, initial data and solutions of the PDE considered are expanded
in a Hilbert basis of $L^{2}(\Omega)$ made of eigenfunctions $(\phi_{j})_{j\geq1}$
of the Laplace operator. The randomisation procedure consists hence
in multiplying all the terms of the series decomposition by well-chosen
independent random variables. In particular, regarding for instance
the homogeneous wave equation with Dirichlet boundary conditions,
one can show that for all initial data $(y^{0},y^{1})\in H_{0}^{1}(\Omega)\times L^{2}(\Omega)$,
the Bernoulli randomisation keeps the $H_{0}^{1}\times L^{2}$ norm
constant. It is observed that many other choices of randomisation
are possible. For instance a positive effect of the randomisation
can be observed by considering independent centred Gaussian random
variables with variance 1 (Gaussian randomisation). Indeed, it allows
to generate a dense subset of the space of initial data $H_{0}^{1}(\Omega)\times L^{2}(\Omega)$
through the mapping 
\[
R_{(y^{0},y^{1})}\colon A\to H_{0}^{1}(\Omega)\times L^{2}(\Omega),\qquad\nu\mapsto(y_{\nu}^{0},y_{\nu}^{1}),
\]
where 
$(y_{\nu}^{0},y_{\nu}^{1})$ denotes the pair of randomised initial
data, provided that all the coefficients in the series expansion of
$(y^{0},y^{1})$ are nonzero. Several other properties of these randomisation
procedures are also established in \cite{BurqTzvetkov}.

\subsection{Randomised observability constant \label{sec:randomObsCst}}

We now review how randomisation appeared in the framework of inverse
problems involving an observability inequality. The property of observability
of a system is related to the following issue: how to recover the
solutions of a PDE from the knowledge of partial measurements of the
solutions. In what follows, we will concentrate on wave models, having
in particular photoacoustic/thermoacoustic tomography imaging in mind.

Let $T>0$ and $\Omega\subseteq\R^{d}$ be a bounded Lipschitz domain
with outer unit normal $\n$. We consider the homogeneous wave equation
with Dirichlet boundary conditions 
\begin{equation}
\left\{ \begin{array}{ll}
\partial_{tt}y(t,x)-\Delta y(t,x)=0 & \quad(t,x)\in[0,T]\times\Omega,\\
y(t,x)=0 & \quad(t,x)\in[0,T]\times\partial\Omega.
\end{array}\right.\label{waveEqobs}
\end{equation}
It is well known that, for all $(y^{0},y^{1})\in H_{0}^{1}(\Omega)\times L^{2}(\Omega)$,
there exists a unique solution $y\in C^{0}([0,T],H_{0}^{1}(\Omega))\cap C^{1}((0,T),L^{2}(\Omega))$
of \eqref{waveEqobs} such that $y(0,x)=y^{0}(x)$ and $\partial_{t}y(0,x)=y^{1}(x)$
for almost every $x\in\Omega$. Let $\Gamma$ be a measurable subset
of $\partial\Omega$, representing the domain occupied by some sensors,
which take some measurements over a time horizon $[0,T]$.

The inverse problem under consideration reads as follows. 
\begin{quote}
\label{eq:ip} \textbf{Inverse problem:} reconstruct the initial condition
$(y^{0},y^{1})$ from the knowledge of the partial boundary measurements
\[
\one_{\Gamma}(x)\frac{\partial y}{\partial\n}(t,x),\qquad(t,x)\in[0,T]\times\partial\Omega.
\]
\end{quote}
To solve this problem, we introduce the so-called \textit{observability
constant}: $C_{T}(\Gamma)$ is defined as the largest non-negative
constant $C$ such that 
\begin{equation}
C\Vert(y(0,\cdot),\partial_{t}y(0,\cdot))\Vert_{H_{0}^{1}(\Omega)\times L^{2}(\Omega)}^{2}\leq\int_{0}^{T}\int_{\Gamma}\left|\frac{\partial y}{\partial\n}(t,x)\right|^{2}\,d\Hn\,dt,\label{ineqobs}
\end{equation}
for any solution $y$ of \eqref{waveEqobs}, where $H_{0}^{1}(\Omega)$
is equipped with the norm $\|u\|_{H_{0}^{1}(\Omega)}=\|\nabla u\|_{L^{2}(\Omega)}$.
Then, the aforementioned inverse problem is well-posed if and only
if $C_{T}(\Gamma)>0$. In such a case, we will say that observability
holds true in time $T$. Moreover, observability holds true within
the class of $\mathcal{C}^{\infty}$ domains $\Omega$ if $(\Gamma,T)$
satisfies the \textit{Geometric Control Condition (GCC)} (see \cite{BLR}),
and this sufficient condition is almost necessary.

Let us express the observability constant more explicitly. Fix an
orthonormal basis (ONB) $(\phi_{j})_{j\geq1}$ of $L^{2}(\Omega)$
consisting of (real-valued) eigenfunctions of the Dirichlet-Laplacian
operator on $\Omega$, associated with the negative eigenvalues $(-\lambda_{j}^{2})_{j\geq1}$.
Then, any solution $y$ of \eqref{waveEqobs} can be expanded as 
\begin{equation}
y(t,x)=\sum_{j=1}^{+\infty}y_{j}(t)\phi_{j}(x)={\frac{1}{\sqrt{2}}}\sum_{j=1}^{+\infty}\left(\frac{a_{j}}{{\lambda_{j}}}e^{i\lambda_{j}t}+\frac{b_{j}}{{\lambda_{j}}}e^{-i\lambda_{j}t}\right)\phi_{j}(x),\label{yDecomp_intro}
\end{equation}
where the coefficients $a_{j}$ and $b_{j}$ account for initial data.
More precisely, we consider the ONB of $H_{0}^{1}(\Omega)\times L^{2}(\Omega)$
given by $\{\psi_{j}^{+},\psi_{j}^{-}:j\ge1\}$, where 
\begin{equation}
\psi_{j}^{+}=\frac{1}{\sqrt{2}}(\frac{\phi_{j}}{\lambda_{j}},i\phi_{j}),\qquad\psi_{j}^{-}=\frac{1}{\sqrt{2}}(\frac{\phi_{j}}{\lambda_{j}},-i\phi_{j}).\label{eq:ONBL2}
\end{equation}
Expanding now the initial data with respect to this basis we can write
$(y^{0},y^{1})=\sum_{j=1}^{+\infty}a_{j}\psi_{j}^{+}+b_{j}\psi_{j}^{-}$,
namely, 
\[
y^{0}=\sum_{j=1}^{+\infty}\frac{a_{j}+b_{j}}{\sqrt{2}\,\lambda_{j}}\,\phi_{j},\qquad y^{1}=\sum_{j=1}^{+\infty}i\frac{a_{j}-b_{j}}{\sqrt{2}}\,\phi_{j}.
\]
The corresponding solution to \eqref{waveEqobs} is given by \eqref{yDecomp_intro}.
In addition, Parseval's identity yields $\|(y^{0},y^{1})\|_{H_{0}^{1}(\Omega)\times L^{2}(\Omega)}^{2}=\sum_{j=1}^{+\infty}|a_{j}|^{2}+|b_{j}|^{2}$.

Then, the constant $C_{T}(\Gamma)$ rewrites 
\[
\begin{split}C_{T}(\Gamma) & =\inf_{\substack{(a_{j}),(b_{j})\in\ell^{2}(\mathds{C})\\
\sum_{j=1}^{+\infty}(|a_{j}|^{2}+|b_{j}|^{2})=1
}
}\int_{0}^{T}\int_{\Gamma}\left|\frac{\partial y}{\partial\n}(t,x)\right|^{2}\,d\Hn\,dt,\end{split}
\]
where $y(t,x)$ is given by \eqref{yDecomp_intro}.

\medskip{}

The constant $C_{T}(\Gamma)$ is {deterministic} and takes into
account any $(a_{j}),(b_{j})\in\ell^{2}(\mathds{C})$, including the
worst possible cases. Interpreting $C_{T}(\Gamma)$ as a quantitative
measure of the well-posed character of the aforementioned inverse
problem, one could expect that such worst cases do not occur too often;
thus it would appear desirable to consider a notion of observation
in average.

Motivated by the findings of Paley and Zygmund (see $\S$\ref{sec:paley})
and its recent use in another context \cite{BurqICM,BurqTzvetkov},
making a random selection of all possible initial data for the wave
equation \eqref{waveEqobs} consists in replacing $C_{T}(\Gamma)$
with the so-called \emph{randomised observability constant} defined
by 
\begin{equation}
C_{T,\mathrm{rand}}(\Gamma)=\hspace{-0.5cm}\inf_{\substack{(a_{j}),(b_{j})\in\ell^{2}(\mathds{C})\\
\sum_{j=1}^{+\infty}(|a_{j}|^{2}+|b_{j}|^{2})=1
}
}\hspace{-0.2cm}\mathds{E}\left(\int_{0}^{T}\int_{\Gamma}\left|\frac{\partial y^{\nu}}{\partial\n}(t,x)\right|^{2}\hspace{-0.1cm}d\Hn\,dt\right),\label{CTrand}
\end{equation}
where 
\begin{equation}
y^{\nu}(t,x)={\frac{1}{\sqrt{2}}}\sum_{j=1}^{+\infty}\left(\frac{\beta_{1,j}^{\nu}a_{j}}{{\lambda_{j}}}e^{i\lambda_{j}t}+\frac{\beta_{2,j}^{\nu}b_{j}}{{\lambda_{j}}}e^{-i\lambda_{j}t}\right)\phi_{j}(x)\label{def:ynu}
\end{equation}
and $(\beta_{1,j}^{\nu})_{j\in\N}$ and $(\beta_{2,j}^{\nu})_{j\in\N}$
are two sequences of independent random variables of Bernoulli or
Gaussian type, on a probability space $(A,\mathcal{A},\mathds{P})$
with mean 0 and variance 1. Here, $\mathds{E}$ is the expectation
in the probability space, and runs over all possible events $\nu$.
In other words, we are randomising the Fourier coefficients $\{a_{j},b_{j}\}_{j\ge1}$
of the initial data $(y^{0},y^{1})$ with respect to the basis $\{\psi_{j}^{\pm}\}_{j\ge1}$.

The randomised observability constant was introduced in \cite{PTZ1,PTZparab,PTZobsND,MR3502963,PTZobsbound}. It can be expressed in terms of
deterministic quantities (see \cite[Theorem~2.2]{PTZobsND}). 
\begin{proposition}
\label{propHazardCst} Let $\Gamma\subset\partial\Omega$ be measurable.
We have 
\begin{equation}
C_{T,\mathrm{rand}}(\Gamma)={\frac{T}{2}}\inf_{j\in\N}\frac{1}{\lambda_{j}^{2}}\int_{\Gamma}\left(\frac{\partial\phi_{j}}{\partial\n}(x)\right)^{2}\,d\Hn.
\end{equation}
 
\end{proposition}

\begin{proof}
In view of Lemma~\ref{lem:minimal} (see below), we have that 
\[
C_{T,\mathrm{rand}}(\Gamma)=\inf\{\|\one_{\Gamma}\partial_{\n}y_{j}^{+}\|_{L^{2}([0,T]\times\partial\Omega)}^{2},\|\one_{\Gamma}\partial_{\n}y_{j}^{-}\|_{L^{2}([0,T]\times\partial\Omega)}^{2}:j\ge1\},
\]
where $y_{j}^{\pm}$ is the solution to \eqref{waveEqobs} with initial
condition $\psi_{j}^{\pm}$. Thus 
\[
y_{j}^{\pm}(t,x)=\frac{1}{\sqrt{2}\,\lambda_{j}}e^{\pm i\lambda_{j}t}\phi_{j}(x),
\]
and in turn 
\[
\|\one_{\Gamma}\partial_{\n}y_{j}^{\pm}\|_{L^{2}([0,T]\times\partial\Omega)}^{2}=\frac{1}{2\lambda_{j}^{2}}\int_{[0,T]\times\Gamma}|e^{\pm i\lambda_{j}t}\partial_{\n}\phi_{j}(x)|^{2}\,dtdx,
\]
which leads to our thesis. 
\end{proof}
We have $C_{T,\mathrm{rand}}(\Gamma)\geq C_{T}(\Gamma)$ (see Proposition~\ref{prop:basis}
below). It has been noted in \cite{MR3502963} that the observability
inequality defining $C_{T,\mathrm{rand}}(\Gamma)$ is associated to
a deterministic control problem for the wave equation \eqref{waveEqobs},
where the control has a particular form but acts in the whole domain
$\Omega$.

Regarding $C_{T,\mathrm{rand}}(\Gamma)$, we refer to \cite[Section 4]{PTZobsbound}
for a discussion on the positivity of this constant. The authors show
that if $\Omega$ is either a hypercube or a disk, then $C_{T,\mathrm{rand}}(\Gamma)>0$
for every relatively non-empty open subset $\Gamma$ of $\partial\Omega$.
In particular, in some cases $C_{T}(\Gamma)=0$ while $C_{T,\mathrm{rand}}(\Gamma)>0$.
This raised hopes that, even if recovering \textit{all} unknowns is
an unstable process, recovering \textit{most} unknowns could be feasible,
since apparently most unknowns are observable. This heuristic argument,
mentioned amongst the motivations for the study of the optimisation
of the randomised observability constant, was not investigated further
in the aforementioned papers. This matter will be studied in the
following sections, dedicated more generally on the possible use of such a constant for investigating the well-posed character of general inverse problems.

\subsubsection*{Applications to optimal design problems}

A larger observability constant $C_{T}(\Gamma)$ in \eqref{ineqobs}
leads to a smaller Lipschitz norm bound of the inverse map. Therefore
$C_{T}(\Gamma)$ can be used as the quantity to maximise when searching
for optimal sensors' positions. However, this turns out to be somewhat
impractical. When implementing a reconstruction process, one has to
carry out in general a very large number of measures; likewise, when
implementing a control procedure, the control strategy is expected
to be efficient in general, but possibly not for all cases. Thus,
one aims at exhibiting an observation domain designed to be the best
possible in average, that is, over a large number of experiments.
Adopting this point of view, it appeared relevant to consider an \textit{average
over random initial data}. In \cite{PTZ1,PTZparab,PTZobsND}, the
best observation is modelled in terms of maximising a \textit{randomised
observability constant}, which coincides with $C_{T,\mathrm{rand}}(\Gamma)$
when dealing with the boundary observation of the wave equation.

When dealing with internal observation of the wave equation on a closed
manifold, it has been shown in \cite{HPT} that the related observability
constant reads as the minimum of two quantities: the infimum of the
randomised observability constants over every orthonormal eigenbasis
and a purely geometric criterion standing for the minimal average
time spent by a geodesic in the observation set.

However, one should keep in mind that a large randomised constant
may not be associated with a reconstruction method (see section~\ref{sec:propcstab}).

\section{The randomised stability constant for abstract inverse problems\label{sec:randomAbstract}}

It is convenient to generalise the construction of the previous section
to an abstract setting. In what follows, none of the arguments require
the precise form of the forward operator related to the wave equation,
as they rely solely on the structure of the randomised constant.

For the remainder of this paper, we let $X$ and $Y$ be separable
infinite-dimensional Hilbert Spaces, and $P\colon X\to Y$ be an injective
bounded linear operator.

If $P^{-1}\colon\operatorname{ran}P\to X$ is a bounded operator,
the inverse problem of finding $x$ from $P(x)$ can be solved in
a stable way for all $x\in X$, without the need for randomisation.
This can be measured quantitatively by the constant 
\[
\Cdet=\inf_{x\in X\setminus\left\{ 0\right\} }\frac{\left\Vert Px\right\Vert _{Y}^{2}}{\left\Vert x\right\Vert _{X}^{2}}>0.
\]
The smaller $\Cdet$ is, the more ill-conditioned the inversion becomes.

On the other hand, when $P^{-1}$ is unbounded, the situation is different,
and the inverse problem is ill-posed \cite{2011-kirsch,2017-hasanov-romanov}.
In this case, although the kernel of $P$ reduces to $\{0\}$, we
have 
\begin{equation}
\Cdet=\inf_{x\in X\setminus\left\{ 0\right\} }\frac{\left\Vert Px\right\Vert _{Y}^{2}}{\left\Vert x\right\Vert _{X}^{2}}=0.\label{eq:defnPnotobs}
\end{equation}

Examples of such maps abound. The example that motivated our study
was that introduced in section~\ref{sec:randomObsCst}, with $X=H_{0}^{1}(\Omega)\times L^{2}(\Omega)$,
$Y=L^{2}([0,T]\times\partial\Omega)$ and 
\[
P(y^{0},y^{1})=y|_{[0,T]\times\Gamma},
\]
where $\Gamma\subseteq\partial\Omega$ and $y$ is the solution to
\eqref{waveEqobs} with initial condition $(y^{0},y^{1})$: if $\Gamma$
is not large enough, $P$ is still injective but $P^{-1}$ is unbounded
\cite{2019-laurent-leautaud}. Any injective compact linear operator
satisfies \eqref{eq:defnPnotobs}.

Let us now introduce the \textit{randomised stability constant}, which
generalises the randomised observability constant to this general
setting. We consider the class of random variables introduced in the
last section. Choose an ONB $\e=\{\e_{k}\}_{k\in\N}$ of $X$ and
write $x=\sum_{k=1}^{\infty}x_{k}\e_{k}\in X$. We consider random
variables of the form 
\[
x^{\nu}=\sum_{k=1}^{\infty}\beta_{k}^{\nu}x_{k} \e_{k},
\]
where $\beta_{k}^{\nu}$ are i.i.d.\ complex-valued random variables
{on a probability space $(A,\mathcal{A},\mathds{P})$} with vanishing
mean and variance $1$, so that $\mathds{E}(|\beta_{k}^{\nu}|^{2})=1$ for every $k$. These include the Bernoulli and Gaussian
random variables considered in the previous section. It is worth observing that, in the case of Bernoulli random variables, we have $|\beta_{k}^{\nu}|^{2}=1$ for every $k$, so that
$\|x^{\nu}\|_{X}=\|x\|_{X}$.
\begin{definition}
\label{def:Crand}The \textit{randomised stability constant} is defined
as 
\[
\Crand=\inf_{x\in X\setminus\left\{ 0\right\} }\mathds{E}\left(\frac{\left\Vert P\left(x^{\nu}\right)\right\Vert _{Y}^{2}}{\left\Vert x\right\Vert _{X}^{2}}\right).
\]
 
\end{definition}

As in the previous section, this constant should represent stability
of the inverse problem for {\textit{m}ost} unknowns. By definition,
we have $\Crand\ge\Cdet$. In general, this is a strict inequality:
we will provide examples in section~\ref{sec:propcstab}. This can
be heuristically seen also by the following deterministic expression
for the randomised stability constant.
\begin{lemma}
\label{lem:minimal}There holds 
\begin{equation}
\Crand=\inf_{k}\left\Vert P\left(\e_{k}\right)\right\Vert _{Y}^{2}.\label{eq:crand-1}
\end{equation}
\end{lemma}

\begin{proof}
Since $\left(Y,\left\Vert \cdot\right\Vert _{Y}\right)$ is also a
Hilbert space, we find 
\begin{equation}
\begin{split}\left\Vert P\left(x^{\nu}\right)\right\Vert _{Y}^{2} & =\left\langle \sum_{k=1}^{\infty}\beta_{k}^{\nu}x_{k} P\left(\e_{k}\right),\sum_{l=1}^{\infty}\beta_{l}^{\nu}x_{l}P\left(\e_{k}\right)\right\rangle _{Y}\\
 & =\sum_{k=1}^{\infty}|\beta_{k}^{\nu}|^{2} |x_{k}|^{2}\left\Vert P\left(e_{k}\right)\right\Vert _{Y}^{2}+\sum_{\substack{k,l\\
k\neq l
}
}\beta_{k}^{\nu}\overline{\beta_{l}^{\nu}}x_{k}\overline{x_{l}}\left\langle P\left(\e_{k}\right),P\left(\e_{k}\right)\right\rangle_Y.
\end{split}
\label{eq:main}
\end{equation}
Since $\beta_{k}^{\nu}$ are i.i.d.\ with vanishing
mean and such that $\mathds{E}(|\beta_{k}^{\nu}|^{2})=1$, we obtain 
\begin{equation}
\mathds{E}\left(\frac{\left\Vert P\left(x^{\nu}\right)\right\Vert _{Y}^{2}}{\left\Vert x\right\Vert _{X}^{2}}\right)=\frac{\mathds{E}\left(\left\Vert P\left(x^{\nu}\right)\right\Vert _{Y}^{2}\right)}{\sum_{k=1}^{\infty}|x_{k}|^{2}}=\frac{\sum_{k=1}^{\infty}|x_{k}|^{2}\left\Vert P\left(\e_{k}\right)\right\Vert _{Y}^{2}}{\sum_{k=1}^{\infty}|x_{k}|^{2}}\geq\inf_{k}\left\Vert P\left(\e_{k}\right)\right\Vert _{Y}^{2},\label{eq:lem1}
\end{equation}
which means $\Crand\ge\inf_{k}\left\Vert P\left(\e_{k}\right)\right\Vert _{Y}^{2}$.
Choosing $x=\e_{k}$, we obtain (\ref{eq:crand-1}). 
\end{proof}

\section{Exploiting $C_{\mathrm{rand}}$ in inverse problems \label{sec:propcstab}}

The aim of this section is to discuss the impact of the randomised
observability constant in inverse problems. In other words, we would
like to address the following question: how can the positivity of
$C_{\textrm{rand}}$ be exploited in the solution of an inverse ill-posed
problem? We will not fully address this issue, but rather provide
a few positive and negative partial results.

We remind the reader that the randomisation introduced in the last
section when \eqref{eq:defnPnotobs} holds, was based on the point
of view that the ratio in \eqref{eq:defnPnotobs} is not ``usually''
small, and that, hopefully, in most cases inversion ``should'' be
possible. It is worthwhile to observe that the subset of the $x\in X$
such that $\left\Vert Px\right\Vert _{Y}\geq c\left\Vert x\right\Vert _{X}$
for some fixed $c>0$ is never generic in $X$, since $\left\{ x\in X:\left\Vert Px\right\Vert _{Y}<c\left\Vert x\right\Vert _{X}\right\} $
is open and non empty by (\ref{eq:defnPnotobs}). This caveat in mind,
we nevertheless wish to test if some evidence can be given to support
our optimistic approach that in most cases, inversion should be possible.
\begin{proposition}
\label{prop:easy_deviation} For every $\epsilon>0$ and $x\in X$,
there exists $c>0$ such that 
\begin{equation}
\mathds{P}\left(\left\Vert Px^{\nu}\right\Vert _{Y}\geq c\left\Vert x\right\Vert _{X}\right)>1-\epsilon.\label{eq:easy_deviation}
\end{equation}
 
\end{proposition}

\begin{proof}
Take $x\in X$. Define the real-valued map $g(c)=\mathds{P}\left(\left\Vert Px^{\nu}\right\Vert _{Y}\geq c\left\Vert x\right\Vert _{X}\right).$
Take a sequence $c_{n}\searrow0$. It is enough to show that 
\[
\lim_{n\to+\infty}g(c_{n})=1.
\]

We write 
\[
g\left(c_{n}\right)=\int_{A}f_{n}\left(\nu\right)d\mathds{P}(\nu),\text{ where }f_{n}\left(\nu\right)=\begin{cases}
1 & \text{if \ensuremath{\left\Vert Px^{\nu}\right\Vert _{Y}\geq c_{n}\left\Vert x\right\Vert _{X}},}\\
0 & \text{otherwise.}
\end{cases}
\]
Note that $f_{n}$ is monotone increasing and, since $\ker P=\left\{ 0\right\} $,
$\lim_{n\to\infty}f_{n}(\nu)=1$ for every $\nu$. Thus by the Monotone
Convergence Theorem, 
\[
\lim_{n\to\infty}\int_{A}f_{n}\left(\nu\right)d\mathds{P}(\nu)=\int_{A}\left(\lim_{n\to\infty}f_{n}\left(\nu\right)\right)d\mathds{P}(\nu)=1.\qedhere
\]
\end{proof}
This thus shows that, for a fixed $x$, our intuition is vindicated:
in the vast majority of cases, the inequality $\left\Vert Px^{\nu}\right\Vert _{Y}\geq c\left\Vert x\right\Vert _{X}$
holds true. This is true independently of $\Crand$; we now investigate whether the positivity of $\Crand$ may yield a stronger estimate.

\subsection{Large deviation inequalities}

The next step is to estimate the probability that, for a given $x\in X\setminus\{0\}$,
the square of the ratio $\frac{\left\Vert P\left(x^{\nu}\right)\right\Vert _{Y}}{\left\Vert x\right\Vert _{X}}$
used in Definition~\ref{def:Crand} is close to its mean value. The
large deviation result we could derive describes the deviation from
an upper bound to $\Crand$, namely the constant $\Krand$ defined
by 
\begin{equation}
\Krand=\sup_{k}\|P(\e_{k})\|_{Y}^{2}.\label{def:KT}
\end{equation}

\begin{theorem}[large deviation estimate]
\label{theoLDE} Assume that $Y=L^{2}(\Sigma,\mu)$, where $(\Sigma,S,\mu)$
is a measure space. Let $(\beta_{k}^{\nu})_{k\in\N}$ be a sequence
of independent random variables of Bernoulli type, on a probability
space $(A,\mathcal{A},\mathds{P})$ with mean 0 and variance 1. Let
$x\in X\setminus\{0\}$ and $x^{\nu}=\sum_{k=1}^{\infty}\beta_{k}^{\nu}\e_{k}x_{k}$.
Then, for every $\delta>0$ we have 
\[
\mathds{P}\left({\Vert Px^{\nu}\Vert_{Y}\ge\delta\left\Vert x^{\nu}\right\Vert _{X}}\right)\leq\exp\left(2-\frac{1}{e}\frac{\delta}{\sqrt{\Krand}}\right).
\]
\end{theorem}

The proof of Theorem \ref{theoLDE} is postponed to Appendix \ref{appendproof:LDE}.
The argument follows the same lines as the one of \cite[Theorem 2.1]{BurqICM}
and the general method introduced in \cite{BurqTzvetkov}.
\begin{remark}[Application to a wave system]
\label{rk:CTrandKTrand} Considering the wave equation \eqref{waveEqobs}
and adopting the framework of section~\ref{sec:randomObsCst} leads
to choose $\{\psi_{j}^{\pm}\}_{j\geq1}$ defined by \eqref{eq:ONBL2}
as the orthonormal basis $\mathrm{e}$. In that case, $X=H_{0}^{1}(\Omega)\times L^{2}(\Omega)$,
$\Sigma=[0,T]\times\partial\Omega$, $d\mu=dt\,d\Hn$ and $Y=L^{2}(\Sigma)$.
Following the discussion in section~\ref{sec:randomObsCst}, the
map $P$ is given by $P(y^{0},y^{1})=\one_{\Gamma}\frac{\partial y}{\partial\n}$,
where $y$ is the unique solution of \eqref{waveEqobs}. Further,
we have 
\begin{align*}
\Crand & =\frac{T}{2}\inf_{j\in\N}\frac{1}{\lambda_{j}^{2}}\int_{\Gamma}\left(\frac{\partial\phi_{j}}{\partial\n}\right)^{2}\,d\Hn,\\
K_{\text{rand}}\left(\mathrm{e}\right) & =\frac{T}{2}\sup_{j\in\N}\frac{1}{\lambda_{j}^{2}}\int_{\Gamma}\left(\frac{\partial\phi_{j}}{\partial\n}\right)^{2}\,d\Hn,
\end{align*}
where the first equality is given in Proposition~\ref{propHazardCst},
and the second one follows by applying the same argument.

Note that according to the so-called Rellich identity\footnote{This identity, discovered by Rellich in 1940 \cite{Rellich2}, reads
\[
2\lambda^{2}=\int_{\partial\Omega}\langle x,\nu\rangle\left(\frac{\partial\phi}{\partial\n}\right)^{2}\,d\Hn
\]
for every eigenpair $(\lambda,\phi)$ of the Laplacian-Dirichlet operator,
$\Omega$ being a bounded connected domain of $\R^{n}$ either convex
or with a $C^{1,1}$ boundary. }, we have $0<K_{\text{rand}}\left(\mathrm{e}\right)\leq\frac{T}{2}\operatorname{diam}(\Omega)$,
under additional mild assumptions on the domain $\Omega$. 
\end{remark}

The estimate given in Theorem~\ref{theoLDE} is on the ``wrong side'',
since we show that the ratio related to the inversion is much bigger
than $\Krand$ with low probability. The issue
is not the boundedness of $P$, which is given a priori, but of its inverse. This would correspond to a result of the type 
\begin{equation}
\mathds{P}\left(\frac{\|P(x^{\nu})\|_{Y}^{2}}{\|x\|_{X}^{2}}<\Crand-\delta\right)\le\text{small constant},\label{eq:false}
\end{equation}
namely, a quantification of the estimate given in Proposition~\ref{prop:easy_deviation},
uniform in $x$. If such a bound held, it would show that $\Crand$
is a reliable estimator of the behaviour of the ratio $\frac{\|P(x)\|_{Y}^{2}}{\|x\|_{X}^{2}}$
in general. Notice that, in the favourable case when $P^{-1}$ is  bounded, there exists $\delta_0\in [0,\Crand)$ such that
$$
\mathds{P}\left(\frac{\|P(x^{\nu})\|_{Y}^{2}}{\|x\|_{X}^{2}}<\Crand-\delta\right)=0
$$
for all $\delta\in [\delta_0,\Crand)$.

In this general framework, estimate \eqref{eq:false} does not hold,
see Example~\ref{ex:inf}. Using a concentration inequality, a weaker
bound can be derived. 
\begin{proposition}
\label{pro:thm2} Assume that $Y=L^{2}(\Sigma,\mu)$, where $(\Sigma,S,\mu)$
is a measure space. Let $(\beta_{k}^{\nu})_{k\in\N}$ be a sequence
of independent random variables of Bernoulli type, on a probability
space $(A,\mathcal{A},\mathds{P})$ with mean 0 and variance 1. Let
$x\in X\setminus\{0\}$ and $x^{\nu}=\sum_{k=1}^{\infty}\beta_{k}^{\nu}\e_{k}x_{k}$.
Then, for every $\delta>0$ we have 
\begin{equation}
\mathbb{P}\left(\frac{\|P(x^{\nu})\|_{Y}^{2}}{\|x\|_{X}^{2}}-\mathbb{E}\left(\frac{\left\Vert P\left(x^{\nu}\right)\right\Vert _{Y}^{2}}{\left\Vert x\right\Vert _{X}^{2}}\right)<-\delta\right)<\exp\left(-\frac{\delta^{2}}{4\Krand^{2}}\right).\label{eq:lame}
\end{equation}
 
\end{proposition}

This result is based on an appropriate Hoeffding inequality; its proof
is postponed to Appendix~\ref{appendproof:LDE}. Note that \eqref{eq:lame}
is not a large deviation result: the quantity under consideration
is bounded between 0 and 1, and the upper bound obtained is not small.
This is unavoidable, see Example~\ref{ex:inf}.

\subsection{Can you reconstruct two numbers from their sum?}

We collect here several observations that suggest that the positivity
of the randomised stability constant may not be helpful for solving
the inverse problems, not even for most unknowns.

\subsubsection{Instability arises for every $x\in X$}

We remind the reader why (\ref{eq:defnPnotobs}) renders inversion
unstable. Hypothesis (\ref{eq:defnPnotobs}) implies that there exists
a sequence $\left(x_{n}\right)_{n\in\N}$ such that 
\begin{equation}
\left\Vert x_{n}\right\Vert _{X}=1\text{ and }\left\Vert Px_{n}\right\Vert _{Y}<\frac{1}{n}\text{ for all }n\in\N.\label{eq:seqnxn}
\end{equation}
Suppose that our measurements are not perfect, and are affected by
a low level of noise $\delta$, $\left\Vert \delta\right\Vert _{Y}\leq\epsilon$,
with $\epsilon>0$. Then, for every $n$ such that $n\epsilon>1$,
we have 
\[
\left\Vert P\left(x+x_{n}\right)-P\left(x\right)\right\Vert _{Y}<\epsilon,
\]
hence $x$ and $x+x_{n}$ correspond to the same measured data, even
if $\left\Vert \left(x+x_{n}\right)-x\right\Vert _{X}=1$. This is
an unavoidable consequence of the unboundedness of $P^{-1}$, and
is true for every $x\in X$, even if the randomised stability constant
were positive (and possibly large).

\subsubsection{The dependence of $C_{\mathrm{rand}}$ on the basis\label{sec:CrandIntrin}}

Lemma~\ref{lem:minimal} shows that the ONB used to randomise our
input plays a role, as it appears explicitly in the formula \eqref{eq:crand-1}.
The following proposition underscores that point. Namely, if we consider
all possible randomisations with respect to all ONB of $X$ we recover
the deterministic stability constant $\Cdet$. 
\begin{proposition}
\label{prop:basis} We have 
\[
\inf_{\e}\Crand=\Cdet,
\]
where the infimum is taken over all ONB of $X$. In particular, if
$P^{-1}$ is unbounded then $\inf_{\e}\Crand=0$. 
\end{proposition}

\begin{proof}
By definition of $\Crand$, we have that $\Crand\ge\Cdet$ for every
ONB $\e$, and so it remains to prove that 
\[
\inf_{\e}\Crand\le\Cdet.
\]
By definition of $\Cdet$, we can find a sequence $x_{n}\in X$ such
that $\|x_{n}\|_{X}=1$ for every $n$ and $\|Px_{n}\|_{Y}^{2}\to\Cdet.$
For every $n$, complete $x_{n}$ to an ONB of X, which we call $\e^{(n)}$.
By Lemma~\ref{lem:minimal} we have $C_{\mathrm{rand}}(\e^{(n)})\le\|Px_{n}\|_{Y}^{2}$,
and so 
\[
\inf_{\e}\Crand\le\inf_{n}C_{\mathrm{rand}}(\e^{(n)})\le\inf_{n}\|Px_{n}\|_{Y}^{2}\le\lim_{n\to+\infty}\|Px_{n}\|_{Y}^{2}=\Cdet.\qedhere
\]
\end{proof}
This result shows that, in general, the randomised stability constant
strongly depends on the choice of the basis. There will always be
bases for which it becomes arbitrarily small when $P^{-1}$ is unbounded.

It is also worth observing that for compact operators, which arise
frequently in inverse problems, the randomised stability constant
is always zero.
\begin{lemma}
\label{lem:compact} If $P$ is compact then $C_{\mathrm{rand}}(\mathrm{e})=0$
for every ONB $\mathrm{e}$ of $X$. 
\end{lemma}

\begin{proof}
Since $\mathrm{e}_{k}$ tends to zero weakly in $X$, by the compactness
of $P$ we deduce that $P(\mathrm{e}_{k})$ tends to zero strongly
in $Y$. Thus, by Lemma~\ref{lem:minimal} we have 
\[
\Crand=\inf_{k}\left\Vert P\left(\e_{k}\right)\right\Vert _{Y}^{2}\leq\lim_{k\to+\infty}\left\Vert P\left(\e_{k}\right)\right\Vert _{Y}^{2}=0,
\]
as desired. 
\end{proof}

\subsubsection{Examples\label{sec:example}}

Let us now consider some examples. The first example is finite dimensional
and the kernel of the operator is not trivial. Note that the definition
of $\Crand$ and all results above, except {Proposition~\ref{prop:easy_deviation}
and} Lemma~\ref{lem:compact}, are valid also in this case, with
suitable changes due to the finiteness of the ONB.
\begin{example}
\label{ex:sum} Choose $X=\mathds{R}^{2}$ and $Y=\mathds{R}$, and
consider the map 
\begin{align*}
S\colon\mathds{R}^{2} & \to\mathds{R}\\
\left(x,y\right) & \mapsto x+y.
\end{align*}
The associated inverse problem can be phrased: find the two numbers
whose sum is given. This problem is ill-posed and impossible to solve.
The deterministic stability constant vanishes 
\[
\inf_{\left(x_{1},x_{2}\right)\in\mathds{R}^{2}\setminus\left\{ \left(0,0\right)\right\} }\frac{\left|x_{1}+x_{2}\right|^{2}}{x_{1}^{2}+x_{2}^{2}}=0,
\]
and $S^{-1}$ does not exist. However, the randomised constant obtained
using the canonical basis is positive. Indeed, $\left|P\left(1,0\right)\right|=\left|P\left(0,1\right)\right|=1$,
therefore 
\[
C_{\text{rand}}\left(\{\left(1,0\right),(0,1)\}\right)=\inf\{1,1\}=1.
\]
The positivity of this constant does not imply the existence of any
useful method to perform the reconstruction of $x$ and $y$ from
$x+y$, even for most $(x,y)\in\R^{2}$.

Had we chosen as orthonormal vectors $\frac{1}{\sqrt{2}}\left(1,1\right)$
and $\frac{1}{\sqrt{2}}\left(1,-1\right)$ , since $\left|P\left(1,-1\right)\right|=0$,
we would have found 
\[
C_{\text{rand}}\left(\frac{1}{\sqrt{2}}\left(1,1\right),\frac{1}{\sqrt{2}}\left(1,-1\right)\right)=0.
\]
\end{example}

One may wonder whether the features highlighted above are due to the
fact that the kernel is not trivial. That is not the case, as the
following infinite-dimensional generalisation with trivial kernel
shows.
\begin{example}
\label{ex:inf} Consider the case when $X=Y=\ell^{2}$, equipped with
the canonical euclidean norm. Let $\mathrm{e}=\{\mathrm{e}_{k}\}_{k=0}^{+\infty}$
denote the canonical ONB of $\ell^{2}$. Take a sequence $\left(\eta_{n}\right)_{n\in\N_{0}}$
such that $\eta_{n}>0$ for all $n$, and $\lim_{n\to\infty}\eta_{n}=0$.
We consider the operator $P$ defined by 
\[
P(\e_{2n})=\e_{2n}+\e_{2n+1},\qquad P(\e_{2n+1})=\e_{2n}+(1+\eta_{n})\e_{2n+1}.
\]
The operator $P$ may be represented with respect to the canonical
basis $\e$ by the block-diagonal matrix 
\[
P=\begin{bmatrix}1 & 1 & 0 & 0 & \cdots\\
1 & 1+\eta_{0} & 0 & 0 & \cdots\\
0 & 0 & 1 & 1 & \cdots\\
0 & 0 & 1 & 1+\eta_{1} & \cdots\\
\vdots & \vdots & \vdots & \vdots & \ddots
\end{bmatrix}.
\]
In other words, $P$ may be expressed as 
\[
P(x)=(x_{0}+x_{1},x_{0}+(1+\eta_{0})x_{1},x_{2}+x_{3},x_{2}+(1+\eta_{1})x_{3},\dots),\qquad x\in\ell^{2}.
\]

We note that $\ker T=\left\{ 0\right\} $ and that its inverse is
given by 
\[
P^{-1}(y)=((1+\eta_{0}^{-1})y_{0}-\eta_{0}^{-1}y_{1},\eta_{0}^{-1}(y_{1}-y_{0}),(1+\eta_{1}^{-1})y_{2}-\eta_{1}^{-1}y_{3},\eta_{1}^{-1}(y_{3}-y_{2}),\dots),
\]
which is an unbounded operator since $\eta_{n}^{-1}\to+\infty$. Given
the block diagonal structure of this map, the inversion consists of
solving countably many inverse problems (i.e., linear systems) of
the form 
\[
\left\{ \begin{array}{l}
x_{2n}+x_{2n+1}=y_{2n},\\
x_{2n}+(1+\eta_{n})x_{2n+1}=y_{2n+1}.
\end{array}\right.
\]
As soon as $\tilde{n}$ is such that $\eta_{\tilde{n}}$ becomes smaller
than the noise level, all the following inverse problems for $n\ge\tilde{n}$
are impossible to be solved, since they reduce to the ``sum of two
numbers'' discussed in Example~\ref{ex:sum}.

Note that 
\[
\|P(\e_{2n})\|_{2}^{2}=2,\qquad\|P(\e_{2n+1})\|_{2}^{2}=1+\left(1+\eta_{n}\right)^{2},
\]
therefore 
\[
\Crand=2.
\]
If we choose instead the rotated orthonormal basis, 
\[
v_{2n}=\frac{1}{\sqrt{2}}\left(\e_{2n}+\e_{2n+1}\right),\quad v_{2n+1}=\frac{1}{\sqrt{2}}\left(\e_{2n}-\e_{2n+1}\right),
\]
then $P\left(v_{2n+1}\right)=-\eta_{n}\frac{1}{\sqrt{2}}\e_{2n+1},$
and so 
\[
C_{\text{rand}}\left(\{v_{k}\}_{k}\right)=\inf_{k}\|P\left(v_{k}\right)\|_{2}^{2}\le\lim_{n\to\infty}\|P\left(v_{2n+1}\right)\|_{2}^{2}=0.
\]

We now turn to \eqref{eq:false} and \eqref{eq:lame}. For some $k\geq0$,
consider $x=\e_{2k}$. Then $P(x^{\nu})={\beta_{2k}^{\nu}}(\e_{2k}+\e_{2k+1})$
and therefore $\|P(x^{\nu})\|_{Y}=\|P(x)\|_{Y}$: there is no deviation
as 
\[
\frac{\|P(x^{\nu})\|_{Y}^{2}}{\|x\|_{X}^{2}}=\mathbb{E}\left(\frac{\|P(x^{\nu})\|_{Y}^{2}}{\|x\|_{X}^{2}}\right)=\Crand=2.
\]
Thus, the probabilities in \eqref{eq:false} and \eqref{eq:lame}
are equal to $0$, and the inequalities are trivial.

However, alternatively, consider $x=\e_{2k}+\e_{2k+1}$. Then $\|P(x^{\nu})\|_{Y}^{2}=4+(2+\eta_{2k})^{2}$
if $\beta_{1}^{\nu}\beta_{2}^{\nu}=1$ and $\|P(x^{\nu})\|_{Y}^{2}=\eta_{2k}^{2}$
if $\beta_{1}^{\nu}\beta_{2}^{\nu}=-1$. Therefore 
\[
\frac{\|P(x^{\nu})\|_{Y}^{2}}{\|x\|_{X}^{2}}-\Crand=\begin{cases}
\frac{(2+\eta_{2k})^{2}}{2} & \text{with probability \ensuremath{\frac{1}{2}},}\\
\frac{\eta_{2k}^{2}}{2}-2 & \text{with probability \ensuremath{\frac{1}{2}}.}
\end{cases}
\]
As a consequence, \eqref{eq:false} cannot be true in general for
every $x$. Similarly, we have 
\[
\frac{\|P(x^{\nu})\|_{Y}^{2}}{\|x\|_{X}^{2}}-\mathds{E}\left(\frac{\|P(x^{\nu})\|_{Y}^{2}}{\|x\|_{X}^{2}}\right)=\begin{cases}
2+\eta_{2k} & \text{with probability \ensuremath{\frac{1}{2}},}\\
-2-\eta_{2k} & \text{with probability \ensuremath{\frac{1}{2}},}
\end{cases}
\]
and the left-hand side of \eqref{eq:lame} can indeed be large for
some $x$.
\end{example}

It is worth observing that a very similar example is considered in
\cite{Maass2019} to show that particular complications may arise
when using neural networks for solving some inverse problems, even
naive and small scale (cfr.\ $\S$\ref{sec:deepL}).

These examples show that considering the observability constant for
a particular basis sheds little light on a potential stable inversion
of the problem in average, and that considering all possible randomisations
leads to the same conclusion as the deterministic case (confirming
Proposition~\ref{prop:basis}).

\subsection{Linear versus nonlinear problems}

The pitfalls we encountered when we tried to make use of the randomised
stability constant all stem from the linearity of the problems we
are considering. The seminal work of Burq and Tzvetkov \cite{BurqTzvetkov07},
which showed existence of solutions in super-critical regimes for
a semilinear problem did not involve tinkering with associated linear
operator (the wave equation); it is the nonlinearity that controlled
the critical threshold. In both compressed sensing and passive imaging
with random noise sources, nonlinearity plays a key role; further,
deep networks are nonlinear maps (cfr.\ Appendix~\ref{sec:rev}).

The naive intuition we discussed earlier, namely, that extreme situations
do not occur often, is more plausible for nonlinear maps where pathological
behaviour is local. 
\begin{example}
As a toy finite-dimensional example, consider the map 
\[
T\colon\R\to\R,\qquad T(x)=x(x-\epsilon)(x+\epsilon),
\]
for some small $\epsilon>0$. Then $T$ can be stably inverted outside
of a region with size of order $\epsilon$, since there the inverse
is continuous. Thus, a random initial data has little chance of falling
precisely in the problematic region. Such a case cannot happen with
linear maps. 
\end{example}

\begin{example}
Let $A\colon H\to H$ be an unbounded linear operator on a Hilbert
space with compact resolvent, so that the spectrum of $A$ is discrete.
Define the nonlinear map 
\[
T\colon H\times[0,1]\to H\times[0,1],\qquad T(x,\lambda)=(Ax+\lambda x,\lambda).
\]
Note that $A+\lambda I$ is invertible with probability $1$ if $\lambda$
is chosen uniformly in $[0,1]$. Thus, if $H\times[0,1]$ is equipped
with a product probability measure whose factor on $[0,1]$ is the
uniform probability, then $x$ may be reconstructed from $T(x)$ with
probability $1$.
\end{example}

\section{Concluding remarks\label{sec:conclusion}}

In this paper we focused on the randomised stability constant for
linear inverse problems, which we introduced as a generalisation of
the randomised observability constant.

We argue that, despite its intuitive and simple definition, the randomised
stability constant has no implications in the practical solution of
inverse problems, even for most unknowns. As the examples provided
show, this may be due to the linearity of the problem. With nonlinear
problems, the situation is expected to be completely different. It
could be that the randomised stability constant is meaningful in the
context of a nonlinear inversion process, involving for example a
hierarchical decomposition \cite{tadmor-2004,modin-nachman-rondi-2019},
but we do not know of results in that direction: this is left for
future research.

\bibliographystyle{plain}
\bibliography{random}

\appendix

\section{Examples of techniques based on randomisation}

\label{sec:rev}

In this appendix we briefly review three different techniques for
solving inverse problems where randomisation plays a crucial role.
We do not aim at providing an exhaustive overview, or at reporting
on the most recent advances, or at discussing the many variants that
have been studied. The examples we present are used to contrast possible
different approaches, and the level of mathematical understanding
associated with them.

\subsection{Compressed sensing}

Since the seminal works \cite{CRT,2006-donoho}, compressed sensing
(CS) has provided a theoretical and numerical framework to overcome
Nyqvist criterion in sampling theory for the reconstruction of sparse
signals. In other words, sparse signals in $\C^{n}$ may be reconstructed
from $k$ discrete Fourier measurements, with $k$ smaller than $n$
and directly proportional to the sparsity of the signal (up to log
factors, see eqn.\ \eqref{eq:log} below). Let us give a quick overview
of the main aspects of CS, in order to show how it fits in the general
framework of section~\ref{sec:intro}. For additional details, the
reader is referred to the book \cite{FR}, and to the references therein.

Given $s\in\N=\{1,2,\dots\}$, let $X$ be the set of $s$-sparse
signals in $\C^{n}$, namely 
\[
X=\{x\in\C^{n}:\#\supp x\le s\}.
\]
Let $F\colon\C^{n}\to\C^{n}$ denote the discrete Fourier transform.
In fact, any unitary map may be considered, by means of the notion
of incoherence \cite{2007-candes}. In any case, the Fourier transform
is a key example for the applications to Magnetic Resonance Imaging
and Computerised Tomography (via the Fourier Slice Theorem). In order
to subsample the Fourier measurements, we consider subsets $S_{a}$
of cardinality $k$ of $\{1,2,\dots,n\}$ and parametrise them with
$a\in\{1,2,\dots,{\binom{n}{k}}\}$. Let $Y=\C^{k}$ and $P_{a}\colon\C^{n}\to\C^{k}$
be the projection selecting the entries corresponding to $S_{a}$.
We then define the measurement map 
\[
T_{a}=P_{a}\circ F\colon X\to\C^{k}.
\]
In other words, $T_{a}$ is the partial Fourier transform, since only
the frequencies in $S_{a}$ are measured, and $\#S_{a}=k\le n$.

Given an unknown signal $x_{0}\in X$, we need to reconstruct it from
the partial knowledge of its Fourier measurements represented by $y:=T_{a}(x_{0})$.
The sparsity of $x_{0}$ has to play a crucial role in the reconstruction,
since as soon as $k<n$ the map $P_{a}\circ F\colon\C^{n}\to\C^{k}$
necessarily has a non-trivial kernel. It is worth observing that sparsity
is a nonlinear condition: if $X$ were a linear subspace of $\C^{n}$,
the problem would be either trivial or impossible, depending on $\ker(P_{a}\circ F)\cap X$.
Thus nonlinearity plays a crucial role here.

The simplest reconstruction algorithm is to look for the sparsest
solution to $T_{a}x=y$, namely to solve the minimisation problem
\[
\min_{x\in\C^{n}}\|x\|_{0}\quad\text{subject to \ensuremath{T_{a}x=y},}
\]
where $\|x\|_{0}=\#\supp x$. However, this problem is NP complex,
and its direct resolution impractical. Considering the convex relaxation
$\|\cdot\|_{1}$ of $\|\cdot\|_{0}$ leads to a well-defined minimisation
problem 
\begin{equation}
\min_{x\in\C^{n}}\|x\|_{1}\quad\text{subject to \ensuremath{T_{a}x=y},}\label{eq:l1}
\end{equation}
whose solution may be easily found by convex optimisation (in fact,
by linear programming).

The theory of CS guarantees exact reconstruction. More precisely,
if $\tilde{x}$ is a minimiser of \eqref{eq:l1}, then $\tilde{x}=x_{0}$
with high probability, provided that 
\begin{equation}
k\ge Cs\log n,\label{eq:log}
\end{equation}
and that $a$ is chosen uniformly at random in $\{1,2,\dots,{\binom{n}{k}}\}$
(namely, the subset $S_{a}$ is chosen uniformly at random among all
the subsets of cardinality $k$ of $\{1,2,\dots,n\}$) \cite{CRT}.
In addition, in the noisy case, with measurements of the form $y=T_{a}(x_{0})+\eta$
where $\|\eta\|_{2}\le\epsilon$, by relaxing the equality ``$T_{a}x=y$''
to the inequality ``$\|T_{a}x-y\|_{2}\le\epsilon$'' in \eqref{eq:l1},
one obtains the linear convergence rate $\|x_{0}-\tilde{x}\|_{2}\le C\epsilon$,
namely, the solution is stable.

In summary, CS allows for the stable reconstruction of \textit{all}
sparse signals from partial Fourier measurements, for \textit{most}
choices of the measured frequencies. The corresponding forward map
$T_{a}\colon X\to\C^{k}$ is \textit{nonlinear}, simply because $X$
is not a vector space.

\subsection{Passive imaging with random noise sources}

The material presented in this part is taken from \cite{garnier-2016},
to which the reader is referred for more detailed discussion on this
topic.

A typical multistatic imaging problem is the recovery of some properties
of a medium with velocity of propagation $c(x)>0$ from some measurements
at locations $x_{j}\in\R^{3}$ of the solution $u(t,x)$ of the wave
equation 
\[
\partial_{t}^{2}u(t,x)-c(x)^{2}\Delta u(t,x)=f(t)\delta(x-y),\qquad(t,x)\in\R\times\R^{3},
\]
where $f(t)$ is the source pulse located at $y$. One of the major
applications of this setup is geophysical imaging, where one wants
to recover properties of the structure of the earth from measurements
taken on the surface. Generating sources in this context is expensive
and disruptive. Earthquakes are often used as sources, but they are
rare and isolated events. Yet, noisy signals, as they may be recorded
by seismographers, even if low in amplitude, may be relevant and useful
even in absence of important events.

The key idea is to consider the data generated by random sources (e.g.,
in seismology, those related to the waves of the sea). The equation
becomes 
\[
\partial_{t}^{2}u(t,x)-c(x)^{2}\Delta u(t,x)=n(t,x),\qquad(t,x)\in\R\times\R^{3},
\]
where the source term $n(t,x)$ is a zero-mean stationary random process
that models the ambient noise sources. We assume that its autocorrelation
function is 
\[
\mathds{E}(n(t_{1},y_{1})n(t_{2},y_{2}))=F(t_{2}-t_{1})K(y_{1})\delta(y_{1}-y_{2}),
\]
where $F$ is the time correlation function (normalised so that $F(0)=1$)
and $K$ characterises the spatial support of the sources. The presence
of $\delta(y_{1}-y_{2})$ makes the process $n$ delta-correlated
in space.

The reconstruction is based on the calculation of the empirical cross
correlation of the signals recorded at $x_{1}$ and $x_{2}$ up to
time $T$: 
\[
C_{T}(\tau,x_{1},x_{2})=\frac{1}{T}\int_{0}^{T}u(t,x_{1})u(t+\tau,x_{2})\,dt.
\]
Its expectation is the statistical cross correlation 
\[
\mathds{E}(C_{T}(\tau,x_{1},x_{2}))=C^{(1)}(\tau,x_{1},x_{2}),
\]
which is given by 
\begin{equation}
C^{(1)}(\tau,x_{1},x_{2})=\frac{1}{2\pi}\int_{\R\times\R^{3}}\hat{F}(\omega)K(y)\overline{\hat{G}(\omega,x_{1},y)}\hat{G}(\omega,x_{2},y)e^{-i\omega\tau}\,dtdy,\label{eq:C1}
\end{equation}
where $\hat{\cdot}$ denotes the Fourier transform in time and $G(t,x,y)$
is the time-dependent Green's function. Moreover, $C_{T}$ is a self-averaging
quantity, namely 
\[
\lim_{T\to+\infty}C_{T}(\tau,x_{1},x_{2})=C^{(1)}(\tau,x_{1},x_{2})
\]
in probability.

The role of randomised sources is now clear: from the measured empirical
cross correlation $C_{T}$ with large values of $T$ it is possible
to estimate, with high probability, the statistical cross correlation
$C^{(1)}$. Using \eqref{eq:C1}, from $C^{(1)}(\tau,x_{1},x_{2})$
it is possible to recover (some properties of) the Green function
$G$, which yield useful information about the medium, such as travel
times.

\subsection{Deep Learning in inverse problems for PDE}

\label{sec:deepL}

Convolutional Neural Networks have recently been used for a variety
of imaging and parameter reconstruction problems \cite{2019-haltmeier-book},
including Electrical Impedance Tomography (EIT) \cite{Martin2017,hamilton-2018,2019-eit-deep},
optical tomography \cite{SunFeng-2018}, inverse problems with internal
data \cite{berg-2017}, diffusion problems in imaging \cite{Arridge-2018},
computerised tomography \cite{2017-jin-etal,2018-bubba-etal}, photoacoustic
tomography \cite{2018-deep-pat,2018-hauptmann-etal} and magnetic
resonance imaging \cite{Zhu2018,2018-yang}. In the following brief
discussion, we decided to focus on inverse problems for partial differential
equations (PDE), and in particular on EIT, but similar considerations
are valid for most methods cited above.

Significant improvement has been observed in EIT with deep learning
compared to previous imaging approaches. Let $\Omega\subseteq\R^{d}$,
$d\ge2$, be a bounded Lipschitz domain with outer unit normal $\nu$.
The data in EIT is (a part of) of the Dirichlet-to-Neumann map 
\[
\Lambda_{\sigma}\colon\begin{array}[t]{rcl}
H^{\frac{1}{2}}\left(\partial\Omega\right)/\mathds{R} & \to & H^{-\frac{1}{2}}\left(\partial\Omega\right)/\mathds{R}\\
v & \mapsto & \left.\sigma\nabla u\cdot\nu\right|_{\partial\Omega}
\end{array}
\]
where $u(x)$ denotes the unique solution of the elliptic problem
\[
\left\{ \begin{array}{ll}
\operatorname{div}\left(\sigma(x)\nabla u(x)\right)=0 & x\in\Omega,\\
u(x)=v(x) & x\in\partial\Omega,
\end{array}\right.
\]
and $\sigma(x)>0$ is the unknown conductivity. The experimental data
is usually part of the inverse map, namely the Neumann-to-Dirichlet
map $\Lambda_{\sigma}^{-1}$. In two dimensions, provided that the
electrodes are equally separated on the unit disk, the data may be
modelled by 
\[
T_{N}\Lambda_{\sigma}^{-1}T_{N},
\]
where $T_{N}$ is the $L^{2}$ projection on span$\{\theta\mapsto\cos(n\theta):1\leq n\leq N\}$,
where $N$ is the number of electrodes: it is the partial Fourier
transform limited to the first $N$ coefficients.

Direct neural network inversion approaches suffer from drawbacks alike
direct non-regularised inversion attempts: the output is very sensitive
to measurement errors and small variations. Successful approaches
to Deep Learning EIT \cite{Martin2017,hamilton-2018,2019-eit-deep},
and to other parameter identification problems in \PDE, often involve
two steps.

The first step consists in the derivation of an approximate conductivity
$\sigma$ by a stable, albeit blurry, regularised inversion method.
For instance, in \cite{hamilton-2018} the ``D-bar'' equation is
used, while in \cite{Martin2017} a one-step Gauss-Newton method is
used. In both cases, the output of this step is a representation of
the conductivity coefficient, which depends on the inversion method
used. This first step is deterministic and its analysis is well understood.
The forward problem, relating the conductivity to the Dirichlet-to-Neumann
map, is nonlinear, independently of the inversion algorithm used.
Indeed, the map $\Lambda_{\sigma}$ is a linear operator, but $\sigma\mapsto\Lambda_{\sigma}$
is nonlinear.

The second, post-processing, step uses a neural network to ``deblur''
the image, and in fact restores details that were not identifiable
after the first step.

The second step is not unlike other successful usage of deep-learning
approaches for image classification; in general they are known to
be successful only with very high probability (and in turn for random
unknowns). More precisely, since the findings of \cite{szegedy2013intriguing},
deep networks are known to be vulnerable to so-called ``adversarial
perturbations'' (see the review article \cite{2018-akhtar-mian}
and the references therein). Given an image $x$ that is correctly
classified by the network with high confidence, an adversarial perturbation
is a small perturbation $p$ such that the images $x$ and $y=x+p$
are visually indistinguishable but the perturbed image $y$ is misclassified
by the network, possibly with high confidence too. State-of-the-art
classification networks are successful for the vast majority of natural
images, but are very often vulnerable to such perturbations.

These instabilities are not specific to image classification problems;
they appear in the same way in image reconstruction \cite{2019-vegard-etal}.
In this case, given an image that is well-reconstructed by the network,
it is possible to create another image that is visually indistinguishable
from the original one, but that is not well-reconstructed by the network.

A full mathematical understanding of deep networks is still lacking,
and the reasons of this phenomenon are not fully known. However, the
large Lipschitz constant of the network certainly plays a role, since
it is a sign of potential instability: in order for the network to
be effective, the weights of its linear steps need to be chosen large
enough, and the composition of several layers yields an exponentially
large constant.

\section{Proofs of the large deviation estimates\label{appendproof:LDE}}

The proofs of Theorem~\ref{theoLDE} and Proposition~\ref{pro:thm2}
rest upon a classical large deviation estimate, the so-called Hoeffding
inequality, see e.g. \cite[Prop. 2.2]{BurqICM}, whose proof is recalled
for the convenience of the reader. 
\begin{proposition}
\label{prop:largeDevEst} Let $(\alpha_{n}^{\nu})_{n\geq1}$ be a
sequence of sequence of independent random variables of Bernoulli
type, on a probability space $(A,\mathcal{A},\mathds{P})$ with mean
0 and variance 1. Then, for any $t>0$ and any sequence $(u_{n})_{n\geq1}\in\ell^{2}(\C)$,
we have 
\[
\mathds{P}\left(\sum_{n=1}^{+\infty}\alpha_{n}^{\nu}v_{n}<-t\right)=\mathds{P}\left(\sum_{n=1}^{+\infty}\alpha_{n}^{\nu}v_{n}>t\right)\leq\exp\left(-\frac{1}{2}\frac{t^{2}}{\sum_{n=1}^{+\infty}|v_{n}|^{2}}\right).
\]
 
\end{proposition}

\begin{proof}
There holds $\mathbb{E}\left(\exp\left(\alpha_{n}^{\nu}v_{n}\right)\right)=\mathbb{E}\left(\sum_{k=0}^{\infty}\frac{\left(\alpha_{n}^{\nu}v_{n}\right)^{k}}{k!}\right)=\sum_{k=0}^{\infty}\frac{1}{k!}\mathbb{E}\left(\left(\alpha_{n}^{\nu}v_{n}\right)^{k}\right)$.
All odd powers of $k$ vanish as $\alpha_{n}^{\nu}$ has zero mean
and is symmetrical. Therefore for any $\lambda>0$ 
\[
\begin{split}\mathbb{E}\left(\exp\left(\lambda\alpha_{n}^{\nu}v_{n}\right)\right)%
 & =\sum_{k=0}^{\infty}\frac{1}{\left(2k\right)!}\mathbb{E}\left(\left(\lambda\alpha_{n}^{\nu}v_{n}\right)^{2k}\right)\\
 & \leq\sum_{k=0}^{\infty}\frac{\lambda^{2k}v_{n}^{2k}}{2^{k}\left(k!\right)}\\
 & =\exp\left(\frac{\lambda^{2}}{2}v_{n}^{2}\right).
\end{split}
\]
Applying Chernoff's inequality, we obtain for any $\lambda>0$ 
\begin{align*}
\mathbb{P}\left(\sum_{n=1}^{+\infty}\alpha_{n}^{\nu}v_{n}>t\right) & \le\mathbb{E}[\exp\left(\lambda\left(\sum_{n=1}^{+\infty}\alpha_{n}^{\nu}v_{n}\right)\right)]\exp\left(-\lambda t\right)\\
 & \leq\exp\left(\lambda^{2}\frac{1}{2}\sum_{n=1}^{+\infty}v_{n}^{2}-\lambda t\right).
\end{align*}
Choosing $\lambda=\frac{t}{\sum v_{n}^{2}}$ this yields 
\[
\mathbb{P}\left(\sum_{n=1}^{+\infty}\alpha_{n}^{\nu}v_{n}<-t\right)=\mathbb{P}\left(\sum_{n=1}^{+\infty}\alpha_{n}^{\nu}v_{n}>t\right)\leq\exp\left(-\frac{t^{2}}{2\sum v_{n}^{2}}\right),
\]
as desired. 
\end{proof}
We are now ready to prove Theorem~\ref{theoLDE}. 

\begin{proof}[Proof of Theorem~\ref{theoLDE}]
 Fix $r\geq2$ and set $\mathcal{Y}^{\nu}=\frac{P\left(x^{\nu}\right)}{\left\Vert x^{\nu}\right\Vert _{X}}$.
Markov's inequality yields 

\begin{equation}
\begin{split}\mathds{P}\left(\Vert\mathcal{Y}^{\nu}\Vert_{Y}\ge\delta\right) & =\mathds{P}\left(\Vert\mathcal{Y}^{\nu}\Vert_{Y}^{r}\ge\delta^{r}\right)\leq\frac{1}{\delta^{r}}\mathds{E}\left(\Vert\mathcal{Y}^{\nu}\Vert_{Y}^{r}\right).\end{split}
\label{metz1544}
\end{equation}

Let us denote by $\mathcal{L}_{\nu}^{r}$ the standard Lebesgue space
with respect to the probability measure $d\mathds{P}$. Recall that
$Y=L^{2}(\Sigma,{\mu})$. To provide an estimate of the right-hand
side, notice that 
\begin{equation}
\begin{split}\mathds{E}\left(\Vert\mathcal{Y}^{\nu}\Vert_{Y}^{r}\right) & =\int_{A}\Vert\mathcal{Y}^{\nu}\Vert_{Y}^{r}\,d\mathds{P}(\nu)\\
 & =\int_{A}\left(\int_{\Sigma}|\mathcal{Y}^{\nu}(s)|^{2}\,d{\mu(s)}\right)^{r/2}\,d\mathds{P}(\nu)\\
 & =\Big\Vert\int_{\Sigma}|\mathcal{Y}^{\nu}(s)|^{2}\,d{\mu(s)}\Big\Vert_{\mathcal{L}_{\nu}^{r/2}}^{r/2}\\
 & \leq\left(\int_{\Sigma}\big\Vert|\mathcal{Y}^{\nu}(s)|^{2}\big\Vert_{\mathcal{L}_{\nu}^{r/2}}\,d{\mu(s)}\right)^{r/2}\\
 & =\left(\int_{\Sigma}\big\Vert\mathcal{Y}^{\nu}(s)\big\Vert_{\mathcal{L}_{\nu}^{r}}^{2}\,d{\mu(s)}\right)^{r/2}\\
 & =\big\Vert s\mapsto\Vert\mathcal{Y}^{\nu}(s)\Vert_{\mathcal{L}_{\nu}^{r}}\big\Vert_{Y}^{r}
\end{split}
\label{metz1618}
\end{equation}
by using Jensen's inequality.

Furthermore, for a.e.\ $s\in\Sigma$, we have\footnote{Here, we use that if $X$ denotes a non-negative random variable and
$\varphi\colon\R_{+}\to\R_{+}$, then 
\[
\mathds{E}(\varphi(X))=\int_{0}^{+\infty}\varphi'(u)\mathds{P}(X>u)\,du.
\]
} 
\[
\Vert\mathcal{Y}^{\nu}(s)\Vert_{\mathcal{L}_{\nu}^{r}}^{r}=\int_{A}|\mathcal{Y}^{\nu}(s)|^{r}d\mathds{P}(\nu)=\int_{0}^{+\infty}ru^{r-1}\mathds{P}(|\mathcal{Y}^{\nu}(s)|>u)\,du
\]
and by using Proposition \ref{prop:largeDevEst} and the fact that
$\mathcal{Y}^{\nu}(s)$ reads 
\[
\mathcal{Y}^{\nu}(s)=\frac{\sum_{k=1}^{\infty}\beta_{k}^{\nu}x_{k}\left(P\e_{k}\right)(s)}{\sqrt{\sum_{k=1}^{\infty}|x_{k}|^{2}}},
\]
one gets 
\[
\Vert\mathcal{Y}^{\nu}(s)\Vert_{\mathcal{L}_{\nu}^{r}}^{r}\leq2\int_{0}^{+\infty}ru^{r-1}\exp\left(-\frac{1}{2}\frac{\sum_{k=1}^{\infty}|x_{k}|^{2}u^{2}}{\sum_{k=1}^{+\infty}|\left(P\e_{k}\right)(s)|^{2}|x_{k}|^{2}}\right)\,du.
\]
As a consequence, by using the change of variable 
\[
v=\frac{\sqrt{\sum_{k=1}^{\infty}|x_{k}|^{2}}u}{\sqrt{\sum_{k=1}^{+\infty}|\left(P\e_{k}\right)(s)|^{2}|x_{k}|^{2}}}
\]
we get 
\[
\Vert\mathcal{Y}^{\nu}(s)\Vert_{\mathcal{L}_{\nu}^{r}}^{r}\leq C(r)\left(\frac{\sum_{k=1}^{+\infty}|\left(P\e_{k}\right)(s)|^{2}|x_{k}|^{2}}{\sum_{k=1}^{\infty}|x_{k}|^{2}}\right)^{r/2}
\]
with 
\[
C(r)=2\int_{0}^{+\infty}rv^{r-1}e^{-\frac{1}{2}v^{2}}\,dv.
\]
An elementary computation yields $C(r)<r^{r}$. Therefore, 
\[
\Vert\mathcal{Y}^{\nu}(s)\Vert_{\mathcal{L}_{\nu}^{r}}<\left(r^{2}\frac{\sum_{k=1}^{+\infty}|\left(P\e_{k}\right)(s)|^{2}|x_{k}|^{2}}{\sum_{k=1}^{\infty}|x_{k}|^{2}}\right)^{1/2}.
\]
According to \eqref{metz1618}, we infer that 
\[
\begin{split}\mathds{E}\left(\Vert\mathcal{Y}^{\nu}\Vert_{Y}^{r}\right) & <\Bigg\Vert\left(r^{2}\frac{\sum_{k=1}^{+\infty}|\left(P\e_{k}\right)(\cdot)|^{2}|x_{k}|^{2}}{\sum_{k=1}^{\infty}|x_{k}|^{2}}\right)^{1/2}\Bigg\Vert_{Y}^{r}\\
 & =\Bigg\Vert r^{2}\frac{\sum_{k=1}^{+\infty}|\left(P\e_{k}\right)(\cdot)|^{2}|x_{k}|^{2}}{\sum_{k=1}^{\infty}|x_{k}|^{2}}\Bigg\Vert_{L^{1}(\Sigma)}^{r/2}.
\end{split}
\]
From \eqref{def:KT} and estimate \eqref{metz1544}, we get 
\[
\mathds{P}\left(\Vert\mathcal{Y}^{\nu}\Vert_{Y}\ge\delta\right)<\frac{1}{\delta^{r}}\left(r^{2}\frac{\sum_{k=1}^{+\infty}\Vert P\e_{k}\Vert_{Y}^{2}|x_{k}|^{2}}{\sum_{k=1}^{\infty}|x_{k}|^{2}}\right)^{r/2}\le\left(\frac{\Krand}{\delta^{2}}r^{2}\right)^{r/2},
\]
using the triangle inequality. Minimising the upper bound with respect
to $r$, that is, choosing 
\[
r^{2}=\frac{e^{-2}\delta^{2}}{{\Krand}}
\]
in the inequality above, one finally obtains 
\begin{equation}
\mathds{P}\left(\Vert\mathcal{Y}^{\nu}\Vert_{Y}\ge\delta\right)\leq\exp\left(-\frac{e^{-1}\delta}{\sqrt{{\Krand}}}\right).\label{eq:conclthm1}
\end{equation}
Note that we have assumed $r\geq2$, thus implicitly posited that
${e^{-2}\delta^{2}}\geq4{\Krand}$; we multiply the bound by $\exp(2)$
to cover the other case. 
\end{proof}
We conclude by proving Proposition~\ref{pro:thm2}. 
\begin{proof}[Proof of Proposition~\ref{pro:thm2}]
 As in \eqref{eq:main}, we have 
\begin{equation}
\left\Vert P\left(x^{\nu}\right)\right\Vert _{Y}^{2}=\sum_{k=1}^{\infty}x_{k}^{2}\left\Vert P\left(e_{k}\right)\right\Vert _{Y}^{2}+2\sum_{\substack{k,l\\
k<l
}
}\beta_{k}^{\nu}\beta_{l}^{\nu}x_{k}x_{l}\left\langle P\left(\e_{k}\right),P\left(\e_{l}\right)\right\rangle _{Y}.\label{eq:main2}
\end{equation}
Observing that the family $\alpha_{k,l}^{\nu}=\{\beta_{k}^{\nu}\beta_{l}^{\nu}\}_{\substack{k,l\\
k<l
}
}$ is made of independent Bernoulli variables, equal almost surely to
$-1$ or 1, with $1/2$ as probability of success, we apply Proposition~\ref{prop:largeDevEst}
with 
\[
v_{k,l}=2\frac{x_{k}}{\left\Vert x\right\Vert _{X}}\frac{x_{l}}{\left\Vert x\right\Vert _{X}}\left\langle P\left(\e_{k}\right),P\left(\e_{l}\right)\right\rangle _{Y}
\]
and obtain for all $\delta>0$ 
\[
\mathbb{P}\left(\frac{\|P(x^{\nu})\|_{Y}^{2}}{\|x\|_{X}^{2}}-\mathbb{E}\left(\frac{\left\Vert P\left(x^{\nu}\right)\right\Vert _{Y}^{2}}{\left\Vert x\right\Vert _{X}^{2}}\right)<-\delta\right)<\exp\left(-\frac{\delta^{2}}{2K}\right),
\]
with 
\[
\begin{split}K & =\sum_{\substack{k,l\\
k<l
}
}\left(2\frac{x_{k}x_{l}}{\left\Vert x\right\Vert _{X}^{2}}\left\langle P\left(\e_{k}\right),P\left(\e_{l}\right)\right\rangle _{Y}\right)^{2}\\
 & \leq2\sum_{k,l}\left(\frac{\left|x_{k}\right|\left|x_{l}\right|}{\left\Vert x\right\Vert _{X}^{2}}\left\Vert P\e_{k}\right\Vert _{Y}\left\Vert P\e_{l}\right\Vert _{Y}\right)^{2}\\
 & =2\left(\sum_{k}\frac{\left|x_{k}\right|^{2}}{\left\Vert x\right\Vert _{X}^{2}}\left\Vert P\e_{k}\right\Vert _{Y}^{2}\right)^{2}\\
 & \leq2\Krand^{2}.\qedhere
\end{split}
\]
\end{proof}

\end{document}